\documentclass[3p]{elsarticle}

\usepackage{mathrsfs}
\usepackage{amsfonts}
\usepackage{amsmath}
\usepackage{amsthm}
\usepackage[usenames]{color}
\usepackage{footnote}
\usepackage{longtable}
\usepackage{hyperref}
\usepackage{float}
\usepackage{graphicx}
\usepackage{amssymb}
\usepackage{multicol}
\usepackage{subfig}
\usepackage{caption}

\usepackage{algorithm}
\usepackage{algorithmicx}
\usepackage{algpseudocode}
\floatname{algorithm}{Algorithm}

\algnewcommand\algorithmicswitch{\textbf{switch}}  
\algnewcommand\algorithmiccase{\textbf{case}}  
\algnewcommand\algorithmicassert{\texttt{assert}}  
\algnewcommand\Assert[1]{\State \algorithmicassert(#1)}%
\algdef{SE}[SWITCH]{Switch}{EndSwitch}[1]{\algorithmicswitch\ #1\ \algorithmicdo}{\algorithmicend\ \algorithmicswitch}%
\algdef{SE}[CASE]{Case}{EndCase}[1]{\algorithmiccase\ #1}{\algorithmicend\ \algorithmiccase}%
\algtext*{EndSwitch}%
\algtext*{EndCase}%

\newtheorem{theorem}{Theorem}
\newtheorem{corollary}{Corollary}

\newtheorem{lemma}{Lemma}
\newtheorem{example}{Example}

\journal{XXX}

\begin{document}

\begin{frontmatter}

\title{Further Results on Pure Summing Registers and Complementary Ones}



\author[]{Jianrui Xie\corref{mycorrespondingauthor}}
\cortext[mycorrespondingauthor]{Corresponding author}
\ead{jrxie93@stu.xidian.edu.cn}

\address[mymainaddress]{ISN Laboratory, Xidian University, Xi'an 710071, China}

\begin{abstract}
We decide completely the cycle structure of pure summing register (PSR) and complementary summing register (CSR). Based on the state diagram of CSR, we derive an algorithm to generate de Bruijn cycles from CSR inspired by Tuvi Etzion's publication in 1984. We then point out the limitation in generalizations of extended representation we use in the algorithm proposed, with a proof of the fact that only PSR and CSR contain pure cycles all dividing $n+1$.
\end{abstract}


\end{frontmatter}


\section{Introduction}

A binary de Bruijn sequence of order $n$, also referred as full-length shift-register cycle, is a binary sequence with period $2^n$, in which the $2^n$ possible $n$-consecutive digits are all different. Readers are referred to Golomb \cite{b} and Fredricksen \cite{c} for a comprehensive survey of de Bruijn sequences and algorithms.

The $i$th state of an $n$-stage feedback shift register (FSR) is denoted by $S_i=(a_i, a_{i+1}, \cdots, a_{i+n-1})$, an element of $F_2^n$. We denote the state diagram of an FSR with a feedback function $f(x_1, x_2, \cdots, x_n)$ as $G_{f(n)}$. Therefore, $G_{f(n)}$ consists of $2^n$ vertices corresponding to all $n$-tuples of $F_2^n$. The conjugate $\hat{S_i}$ and the companion $S_i^{\prime}$ of $S_i$ are respectively defined by $\hat{S_i}=(a_i\oplus 1, a_{i+1}, \cdots, a_{i+n-1})$ and $S_i^{\prime}=(a_i, a_{i+1}, \cdots, a_{i+n-1}\oplus 1)$, where $\oplus$ is addition modulo 2.

The feedback function $f(x_1, x_2, \cdots, x_n)$ then induces a next-state operator $\rho: F_2^n \mapsto F_2^n$, under which $\rho S_i=S_{i+1}$. A cycle $C$ of length $l$ is a cyclic sequence of $l$ distinct states, i.e. $C=(S_i, S_{i+1}, \cdots, S_{i+l-1})$, where $S_{i+j+1}=\rho S_{i+j}$ for $j=0, 1, \cdots, l-2$ and $S_i=\rho S_{i+l-1}$. An alternate representation of $C$ is to take the first digit of each state in order into a set, namely $C=(a_ia_{i+1}\cdots a_{i+l-1})$. Two cycles are called to be adjacent if they share a conjugate or companion pair \cite{f}.

Golomb \cite{b} gives the necessary and sufficient condition that an $n$-stage FSR produces pure cycles. That is the feedback function $f$ is nonsingular and can be written as $f(x_1, x_2, \cdots, x_n)=x_1+g(x_2, \cdots, x_n)$ where $g$ is a boolean function of $(n-1)$ variables. In this correspondence our attention will be restricted to nonsingular feedback functions only.

Following \cite{a}, we define the extended representation $E(C)$ of cycle $C$ as an $(n+1)$-length vector $E(C)=(x_1, x_2, \cdots, x_{n+1})$ where $x_{n+1}=f(x_1, x_2, \cdots, x_n)$. Recalling that the Hamming weight of a binary vector $\alpha=(a_1, a_2, \cdots, a_n)$ is $wt(\alpha)=\sum\limits_{i=1}^{n}a_{i}$, we then define the extended weight $W_E(C)$ of $C$ as $W_E(C)=wt(E(C))$.

Following \cite{d}, we let $f_k(x_1, x_2, \cdots, x_n)$ be the restriction of any given boolean function $f(x_1, x_2, \cdots, x_n)$ to the set $\{x=(x_1, x_2, \cdots, x_n) \in F_2^n\ |\ wt(x)=k \}$, where $0 \le k \le n$.

A boolean function $f$ of $n$ variables is said to be symmetric if $f(x_1, x_2, \cdots, x_n)=f(x_{P(1)}, x_{P(2)}, \cdots, x_{P(n)})$ for any permutation $P$ of $\{ 1, 2, \cdots, n\}$. In \cite{e} Canteaut and Videau refer to $v(f)=(v_f(0), v_f(1), \cdots, v_f(n))$ as the simplified value vector of $f$, in which $v_f(i)$ is a mapping from $\{ 0, 1, \cdots, n \}$ to $F_2$ and $f(x)=v_f(wt(x))$ for any $x \in F^n_2$. They also refer to $X_{i,n}$ as the elementary symmetric polynomial of degree $i$ in $n$ variables, viz. $X_{0,n}=1$ and $X_{j,n}=\sum\limits_{1\le i_1<i_2<\cdots<i_j\le n} x_{i_1}x_{i_2}\cdots x_{i_j}$, where $j=1, 2, \cdots, n$. Knowing that a symmetric boolean function $f$ of $n$ variables can be written as $f(x_1, x_2, \cdots, x_n)=\mathop{\oplus}\limits_{i=0}^{n}\lambda_f(i)X_{i,n} \, \, (\lambda_f(i) \in F_2)$, the simplified ANF vector of $f$ is then defined by the $(n+1)$-bit vector $\lambda_f=(\lambda_f(0), \lambda_f(1), \cdots, \lambda_f(n))$.

\section{Cycle structure of PSR and CSR}

In \cite{b}, Golomb gives the numbers of cycles for a pure summing register of length $n-1$ (PSR$_{n-1}$): $f_P(x_1, x_2, \cdots, x_{n-1})=x_1\oplus x_2 \oplus \cdots \oplus x_{n-1}$ and a complementary summing register of length $n-1$ (CSR$_{n-1}$): $f_C(x_1, x_2, \cdots, x_{n-1})=x_1\oplus x_2 \oplus \cdots \oplus x_{n-1} \oplus 1$ respectively by
$$S(n-1)=\frac{1}{2n}\sum_{d\mid n}\phi (d)2^{\frac{n}{d}}+\frac{1}{2n}\sum_{even\ d\mid n}\phi (d)2^{\frac{n}{d}}$$
and
$$S^*(n-1)=\frac{1}{2n}\sum_{odd\ d\mid n}\phi (d)2^{\frac{n}{d}}.$$
We have the next two lemmas.
\begin{lemma}\label{Lemma01}
For {\rm PSR}$_n$: $f_P(x_1, x_2, \cdots, x_n)=x_1 \oplus x_2 \oplus \cdots \oplus x_n,$

(1) For even $n$, where $d\mid (n+1)$ holds, or for odd $n$, where $d\mid (n+1)$ holds and $(n+1)/d$ is odd, the number of cycles with length $d$ is
$$M(d)=\frac{1}{2d}\sum_{d^{\prime}\mid d}\mu (d^{\prime})2^{\frac{d}{d^{\prime}}}+\frac{1}{2d}\sum_{even\ d^{\prime}\mid d}\mu (d^{\prime})2^{\frac{d}{d^{\prime}}}.$$

(2) For odd $n$, where $d\mid (n+1)$ holds and $(n+1)/d$ is even, the number of cycles with length $d$ is
$$M(d)=\frac{1}{d}\sum_{d^{\prime}\mid d}\mu (d^{\prime})2^{\frac{d}{d^{\prime}}}.$$
\end{lemma}

\begin{proof}
When $d\mid (n+1)$ holds, assume that there is a state $(a_1,a_2,\cdots,a_n)\in {\rm PSR}_n$ on a cycle whose length is a factor of $d$, we have $(a_1,a_2,\cdots,a_n,a_{n+1})=(a_{d+1},a_{d+2},\cdots,a_{d+n},a_{d+n+1})$. It implies that $a_i=a_{d+i}=a_{2d+i}=\cdots=a_{(\frac{n+1}{d}-1)d+i}$ for $i=1,2,\cdots,d.$ Thus, the $(n+1)$-length vector $(a_1,a_2,\cdots,a_n,a_{n+1})$ consists of $\frac{n+1}{d}$ sections $(a_1,a_2,\cdots,a_d)$ with length $d$, which are linked end to end:
\begin{equation}\label{Equation1}
(a_1,a_2,\cdots,a_n,a_{n+1})=(\overbrace{\overbrace{a_1,a_2,\cdots,a_d}^{d},\overbrace{a_1,a_2,\cdots,a_d}^{d},\cdots,\overbrace{a_1,a_2,\cdots,a_d}^{d}}^{\frac{n+1}{d}})
\end{equation}
Assume that $(a_1,a_2,\cdots,a_d)$ is any $d$-length vector on $F_2$, then the state $(a_1,a_2,\cdots,a_n)$ from Equation (\ref{Equation1}) belongs to {\rm PSR}$_n$ only when
$(\frac{n+1}{d})wt(a_1,a_2,\cdots,a_d)$ is even, where $wt(a_1,a_2,\cdots,a_d)\in \{0,1,2,\cdots,d\}.$

Case one: When $(n+1)$ is odd, i.e. $n$ is even, $\frac{n+1}{d}$ must be odd and $wt(a_1,a_2,\cdots,a_d)$ must be even. All states $(a_1,a_2,\cdots,a_n)$ corresponding to Equation (\ref{Equation1}) must belong to those cycles of {\rm PSR}$_n$ whose lengths are factors of $d$. Therefore, in $G_{f_P(n)}$ there are
$$\binom{0}{d}+\binom{2}{d}+\binom{4}{d}+\cdots+\binom{2\lfloor \frac{d}{2}\rfloor}{d}=2^{d-1}$$
states on those cycles with length dividing $d$. In addition, in $G_{f_P(n)}$ those cycles whose lengths are factors of $d$ contain $\sum\limits_{d^{\prime}\mid d}d^{\prime}M(d^{\prime})$ states. Naturally, we have $\sum\limits_{d^{\prime}\mid d}d^{\prime}M(d^{\prime})=2^{d-1}$, and the solution is $M(d)=\frac{1}{2d}\sum\limits_{d^{\prime}\mid d}\mu (d^{\prime})2^{\frac{d}{d^{\prime}}}.$

Case two: When $(n+1)$ is even, i.e. $n$ is odd, let $n+1=2^km\ (m\ is\ odd)$ and discuss two sub-cases.

(1)  When $(n+1)/d$ is even, $wt(a_1,a_2,\cdots,a_d)$ can be any value of $\{0,1,2,\cdots,d\}$. Let $d=2^{k_1}m_1$, where $m_1\mid m$ and $k_1=0,1,2,\cdots,k-1$, we have $\sum\limits_{d^{\prime}\mid d}d^{\prime}M(d^{\prime})=2^d$, the solution of which is $M(d)=\frac{1}{d}\sum\limits_{d^{\prime}\mid d}\mu (d^{\prime})2^{\frac{d}{d^{\prime}}}.$

(2) When $(n+1)/d$ is odd, let $d=2^km_1$, where $m_1\mid m$. It is easy to prove that
$$M(d)=\frac{1}{2d}\sum_{d^{\prime}\mid d}\mu (d^{\prime})2^{\frac{d}{d^{\prime}}}+\frac{1}{2d}\sum_{even\ d^{\prime}\mid d}\mu (d^{\prime})2^{\frac{d}{d^{\prime}}}.$$
\end{proof}

\begin{lemma}\label{Lemma02}
For {\rm CSR}$_n$: $f_C(x_1, x_2, \cdots, x_n)=1 \oplus x_1 \oplus x_2 \oplus \cdots \oplus x_n$,

(1) For even $n$, where $d\mid (n+1)$ holds, the number of cycles with length $d$ is
$$M^*(d)=\frac{1}{2d}\sum_{d^{\prime}\mid d}\mu (d^{\prime})2^{\frac{d}{d^{\prime}}}.$$

(2) For odd $n$, let $n+1=2^km$, where $m$ is odd and $d\mid m$ holds, the number of cycles with length $d$ is
$$M^*(2^kd)=\frac{1}{2^{k+1}d}\sum_{d^{\prime}\mid d}\mu (d^{\prime})2^{2^k\frac{d}{d^{\prime}}}.$$
Equivalently, for odd $n$, where $d\mid (n+1)$ holds and $\frac{n+1}{d}$ is odd, the number of cycles with length $d$ is
$$M^*(d)=\frac{1}{2d}\sum_{odd\ d^{\prime}\mid d}\mu (d^{\prime})2^{\frac{d}{d^{\prime}}}.$$
\end{lemma}

\begin{proof}
Same assumptions as those in the proof of Lemma \ref{Lemma01}. The state $(a_1,a_2,\cdots,a_n)$ from Equation (\ref{Equation1}) belongs to {\rm CSR}$_n$ only when
$(\frac{n+1}{d})wt(a_1,a_2,\cdots,a_d)$ is odd, where $wt(a_1,a_2,\cdots,a_d)\in \{0,1,2,\cdots,d\}.$

Case one: When $(n+1)$ is odd, i.e. $n$ is even, $\frac{n+1}{d}$ must be odd and $wt(a_1,a_2,\cdots,a_d)$ must be odd, too. All states $(a_1,a_2,\cdots,a_n)$ corresponding to Equation (\ref{Equation1}) must belong to those cycles of {\rm CSR}$_n$ whose lengths are factors of $d$. Therefore, in $G_{f_C(n)}$ there are
$$\binom{1}{d}+\binom{3}{d}+\binom{5}{d}+\cdots+\binom{2\lfloor \frac{d}{2}\rfloor+1}{d}=2^{d-1}$$
states on those cycles with length dividing $d$. In addition, in $G_{f_C(n)}$ those cycles whose lengths are the factors of $d$ contain $\sum\limits_{d^{\prime}\mid d}d^{\prime}M(d^{\prime})$ states. Thus, $\sum\limits_{d^{\prime}\mid d}d^{\prime}M(d^{\prime})=2^{d-1}$ holds and we get $M^*(d)=\frac{1}{2d}\sum\limits_{d^{\prime}\mid d}\mu (d^{\prime})2^{\frac{d}{d^{\prime}}}.$

Case two: When $(n+1)$ is even, i.e. $n$ is odd, $\frac{n+1}{d}$ must be odd as well as $wt(a_1,a_2,\cdots,a_d)$. Let $n+1=2^km$, where $m$ is odd, we have $d=2^km_0$, where $m_0\mid m$. There are $2^{2^km_0-1}$ states of CSR$_n$ belonging to some cycles whose lengths are factors of $d$. We denote their lengths as $t$ and $t=2^km^{\prime}$ ($m^{\prime}\mid m_0$) must hold. In sum, in $G_{f_C(n)}$ there are $\sum\limits_{m^{\prime}\mid m_0}2^km^{\prime}M_2^*(2^km^{\prime})$ states on those cycles with length $t$. Therefore, we have
$$\sum_{m^{\prime}\mid m_0}2^km^{\prime}M_2^*(2^km^{\prime})=\frac{1}{2}2^{2^km_0}$$
with a solution
\begin{equation}\label{Equation2}
M_2^*(2^km_0)=\frac{1}{2^{k+1}m_0}\sum_{m^{\prime}\mid m_0}\mu (m^{\prime})2^{2^k\frac{m_0}{m^{\prime}}}.
\end{equation}

Since $d=2^km_0$, the traversal of $m^{\prime}$ with the constraint of $m^{\prime}\mid m_0$ is equivalent to the traversal of all odd factors of $d$ as well as the traversal of $d^{\prime}$ under the condition of $odd~d^{\prime}\mid d$. Then we have the equivalence between $m_0/m^{\prime}$ and $d/(2^kd^{\prime})$. From Equation (\ref{Equation2}), we have
\begin{align*}
M_2^*(d)&=\frac{1}{2d}\sum_{odd\ d^{\prime}\mid d}\mu (d^{\prime})2^{2^k\frac{d}{2^kd^{\prime}}}\\
&=\frac{1}{2d}\sum_{odd\ d^{\prime}\mid d}\mu (d^{\prime})2^{\frac{d}{d^{\prime}}}.
\end{align*}
Noted that $d=2^km_0$ ($m_0\mid m$), which is equivalent to the traversal of all $d$ satisfying $(n+1)/d$.
\end{proof}

\section{A generation algorithm of de Bruijn cycles from CSR}

\begin{lemma}\label{Lemma1}
For any cycle $C$ in {\rm CSR}$_n$ with length $l$, there are at most $l$ extend representations $E(C)$ and $W_E(C)=2k+1$ always holds, where $k$ is fixed and $0\le k\le \lfloor {\frac {n}{2}} \rfloor$. For each state $S$ on a given cycle $C$, we have $W_E(C)-1\le wt(S)\le W_E(C)$.
\end{lemma}

\begin{proof}
The conclusion is obvious from the fact $W_E(C)=x_0+x_1+\cdots +x_{n-1}+(x_0\oplus x_1 \oplus \cdots \oplus x_{n-1}\oplus 1)$.
\end{proof}

Following \cite{a}, we refer to a cycle $C$ as a run-cycle if all the ones in $E(C)$ form a cyclic run and define the preferred state $P(C)$ of each cycle $C$ in CSR$_n$ as follows. For a run-cycle $C$, $P(C)=(1^{2k+1}0^{n-2k-1})$. For  a non-run-cycle $C$, a unique extended representation $E^*(C)$ is an $(n+1)$-length vector $[0^r1^t0a_1\cdots a_{n-r-t-2}10]$ $(r\ge 0)$ where $t$ is the length of the longest cyclic run of ones and $E^*(C)$ is the largest in base-2 notation among all $(n+1)$-tuples in this form. Then the preferred state is $P(C)=(0^r1^t0a_1\cdots a_{n-r-t-2}1)$.

\begin{lemma}\label{Lemma2}
For any non-run-cycle $C_1$ in {\rm CSR$_n$}, let $P(C_1)=(0^r1^{t_1}0a_1\cdots a_{n-r-t_1-2}1)$, we have $B=(10^r1^{t_1}0a_1\cdots a_{n-r-t_1-2})$ and $P(C_1)^\prime$ are on another cycle $C_2$ with $W_E(C_2)=W_E(C_1)$. Moreover, let $t_2$ be the length of the longest cyclic run of ones in $C_2$, we have either $t_2=t_1+1$ or $t_2=t_1$ and $\lvert P(C_2) \rvert>\lvert P(C_1) \rvert$.
\end{lemma}

\begin{lemma}\label{Lemma3}
Let a state $U_{2k+1}=(u_1, u_2, \cdots, u_{n-1}, 1)$ of a cycle $C_1$ in {\rm CSR$_n$} with $wt(U_{2k+1})+1=W_E(C_1)=2k+1$, i.e. $wt(U_{2k+1})=2k (k\ge 1)$, then its companion $U_{2k+1}^{\prime}$ is on another cycle $C_2$ with $W_E(C_2)=2k-1$.
\end{lemma}

Proofs of Lemma \ref{Lemma2} and Lemma \ref{Lemma3} refer to \cite{a}.

For ease of notations, we define the parity of a vector $S_i=(a_i, a_{i+1}, \cdots, a_{i+n-1})$ as $p_i=a_i\oplus a_{i+1}\oplus \cdots \oplus a_{i+n-1}$.

\begin{lemma}\label{Lemma4}
Let $S_i=(a_i, a_{i+1}, \cdots, a_{i+n-1})$ be the predecessor of either $U_{2k+1}$ or $U_{2k+1}^{\prime}$, where $U_{2k+1}=(u_1, u_2, \cdots, u_{n-1}, 1)$ and $wt(U_{2k+1})=2k$, then $p_i=a_i\oplus 1$.
\end{lemma}

\begin{proof}
With $wt(a_{i+1}, \cdots, a_{i+n-1})=2k-1$ and the definition of $p_i$, the proof is complete.
\end{proof}

\begin{lemma}\label{Lemma5}
Let $S_i=(a_i, a_{i+1}, \cdots, a_{i+n-1})$ be a state of cycle $C$ in {\rm CSR$_n$}. If $p_i=a_i\oplus 1$, then $S_i$ is not $P(C)$ or $P(C)^\prime$.
\end{lemma}

\begin{proof}
Note that $a_{n+1}=a_i$ since $p_i=a_i\oplus 1$. Assume $S_{i+1}=P(C)=(0^r1^t0a_s\cdots a_{n-r-t+s-3}1)$, then $a_{i+n}=a_i=1$. We have $E(C)=(10^r1^t0a_s\cdots a_{n-r-t+s-3}1)$ from $S_i=(10^r1^t0a_s\cdots a_{n-r-t+s-3})$. Let $wt(P(C))=W_{E^*}(C)=2k+1$ by Lemma \ref{Lemma1}, then $W_E(C)=2k+2$ brings a contradiction. Assume $S_{i+1}=P(C)^\prime=(0^r1^t0a_s\cdots a_{n-r-t+s-3}0)$, then $S_{i+n}=a_i=0$ and $p_i=1$. Let $wt(S_i)=2k+1$, and thus $W_{E^*}(C)=wt(P(C))=wt(P(C)^\prime)+1=wt(S_{i+1})+1=2k+2$, which goes against Lemma \ref{Lemma1}.
\end{proof}

For $0\le k\le \lfloor {\frac {n}{2}} \rfloor$, there exists a unique run-cycle with extended weight $2k+1$. Without loss of generality, we set it as the initial main cycle. We continuously join one of the rest of cycles with extended weight $2k+1$ into the main cycle in this order: the cycle has the longest run of ones in the rest and has the largest preferred state if there are more than one cycles having the same longest run of ones. Lemma \ref{Lemma2} guarantees that the two cycles are adjacent and can be joined together.

Be Lemma \ref{Lemma3} cycle-joining method can be applied to $MC_k$ and $MC_{k-1}$ ($1\le k\le \lfloor {\frac {n}{2}} \rfloor$) if we choose such a satisfactory state $U_{2k+1}$ on $MC_k$. In other words, it is feasible to join all $MC_k$ ($0\le k\le \lfloor {\frac {n}{2}} \rfloor$) into the longest one by choosing proper $U_{2k+1}$ for each $1\le k\le \lfloor {\frac {n}{2}} \rfloor$.

\begin{example}\label{Example1}
By Lemma {\rm \ref{Lemma02}}, {\rm CSR}$_7$ consists of $16$ pure cycles with length $8$, including one cycle with extended weight $1$, seven cycles with extended weight $3$, seven cycles with extended weight $5$ and one cycle with extended weight $7$. To be concise, we adopt decimal numbers from $1$ to $128$ to represent states in order from $(0,0,0,0,0,0,0)$ to $(1,1,1,1,1,1,1)$. Therefore, $MC_0$ consists of all cycles with extended weight $1$, i.e. $MC_0=\{1,2,3,5,9,17,33,65\}$. By applying cycle-joining method to all cycles with extended weight $3$ with these state pairs $(P(C),P(C)^{\prime})$ in order: $(98,97)$, $(50,49)$, $(26,25)$, $(14,13)$, $(82,81)$, $(74,73)$, we have $MC_1=$

\begin{center}
\begin{tabular}{  p{16cm}}
\textnormal{\{4, 8, 15, 29, 57, 113, 98, 67, 16, 12, 23, 45, 89, 50, 99, 69, 10, 20, 39, 77, 26, 51, 101, 74, 19, 38, 75, 21, 42, 83, 37, 73, 18, 36, 71, 14, 27, 53, 105, 82, 35, 70, 11, 22, 43, 85, 41, 81, 34, 68, 7, 13, 25, 49, 97, 66\}};
\end{tabular}
\end{center}

Similarly, by using these states pairs $(P(C),P(C)^{\prime})$ in order: $(122,121)$, $(62,61)$, $(116,115)$, $(118,117)$, $(60,59)$, $(110,109)$, we obtain $MC_2=$

\begin{center}
\begin{tabular}{  p{16cm}}
\textnormal{\{16, 32, 63, 125, 122, 116, 103, 78, 28, 56, 111, 93, 58, 115, 102, 76, 24, 48, 95, 62, 123, 118, 107, 86, 44, 88, 47, 94, 60, 119, 110, 91, 54, 108, 87, 46, 92, 55, 109, 90, 52, 104, 79, 30, 59, 117, 106, 84, 40, 80, 31, 61, 121, 114, 100, 72\}};
\end{tabular}
\end{center}
Naturally, $MC_3=\{64,128,127,126,124,120,112,96\}$. By choosing $(U^7,{U^{7}}^{\prime})=(126,125)$, $(U^5,{U^{5}}^{\prime})=(114,113)$ and $(U^3,{U^{3}}^{\prime})=(66,65)$, it is viable to join $MC_2$, $MC_1$ and $MC_0$ into $MC_3$ in order and the longest cycle shall be reached as
\begin{center}
\begin{tabular}{  p{16cm}}
\textnormal{\{64, 128, 127, 125, 122, 116, 103, 78, 28, 56, 111, 93, 58, 115, 102, 76, 24, 48, 95, 62, 123, 118, 107, 86, 44, 88, 47, 94, 60, 119, 110, 91, 54, 108, 87, 46, 92, 55, 109, 90, 52, 104, 79, 30, 59, 117, 106, 84, 40, 80, 31, 61, 121, 113, 98, 67, 16, 12, 23, 45, 89, 50, 99, 69, 10, 20, 39, 77, 26, 51, 101, 74, 19, 38, 75, 21, 42, 83, 37, 73, 18, 36, 71, 14, 27, 53, 105, 82, 35, 70, 11, 22, 43, 85, 41, 81, 34, 68, 7, 13, 25, 49, 97, 65, 1, 2, 3, 5, 9, 17, 33, 66, 4, 8, 15, 29, 57, 114, 100, 72, 16, 32, 63, 126, 124, 120, 112, 96\}};
\end{tabular}
\end{center}
\end{example}

Now we present an algorithm to generate de Bruijn cycles from CSR$_n$.
\begin{algorithm}
        \caption{De Bruijn sequences from CSR}
        \label{Algorithm1}
        \begin{algorithmic}[0] 
            \Require $U_{2k+1}$ (for each $1\le k\le \lfloor {n/2} \rfloor$), $S_1=(a_1,a_2,\cdots,a_n)$, $p_1$ (the parity of $S_1$)
            \Ensure an n-order de Bruijn sequence with $S_1$ as the initial state
            
            \State $i\gets 1$
            \Repeat \State{            
            \Switch{$p_i\oplus a_i$}           
            \Case{$1$}
            \If{$(a_{i+1}, \cdots, a_{i+n-1}, 1)=U_{w_i-a_i+2}$}
            \State $(a_{i+n}, p_{i+1}, w_{i+1})\gets$ \Call{NextBitIntCha}{$a_i, p_i, w_i$}
            \Else
            \State $(a_{i+n}, p_{i+1}, w_{i+1})\gets$ \Call{NextBitStab}{$a_i, p_i, w_i$}     
            \EndIf
            \State $break$
            \EndCase            
            \Case{$0$}
            \If{$S_i^*=[a_{i+1} \cdots a_{i+n-1}10]$is a run-cycle}
            \State $(a_{i+n}, p_{i+1}, w_{i+1})\gets$ \Call{NextBitStab}{$a_i, p_i, w_i$}
            \State $break$
            \EndIf
            \State find the preferred state $P(C)$ of $C=S_i^*$
            \If{$P(C)=[a_{i+1} \cdots a_{i+n-1}1]$}
            \State $(a_{i+n}, p_{i+1}, w_{i+1})\gets$ \Call{NextBitIntCha}{$a_i, p_i, w_i$}
            \Else
            \State $(a_{i+n}, p_{i+1}, w_{i+1})\gets$ \Call{NextBitStab}{$a_i, p_i, w_i$}
            \EndIf
            \State $break$
            \EndCase
            \EndSwitch
            \State $i\gets i+1$
            \State $S_i\gets (a_i, a_{i+1},\cdots, a_{i+n-1})$
            } \Until{$S_i=S_1$}        
            
            \Function{NextBitStab}{$a_i, p_i, w_i$}
                \State $a_{i+n}\gets p_i\oplus 1$
                \State $p_{i+1}\gets a_i\oplus 1$
                \State $w_{i+1}\gets w_i-a_i+(p_i\oplus 1)$
                \State \Return{$(a_{i+n}, p_{i+1}, w_{i+1})$}
             \EndFunction
             
              \Function{NextBitIntCha}{$a_i, p_i, w_i$}
                \State $a_{i+n}\gets p_i$
                \State $p_{i+1}\gets a_i$
                \State$w_{i+1}\gets w_i-a_i+p_i$
                \State \Return{$(a_{i+n}, p_{i+1}, w_{i+1})$}
             \EndFunction
               
        \end{algorithmic}
\end{algorithm}

We provide explanations about the algorithm.

The \textbf{switch} statement examines the value of $p_i\oplus a_i$. By Lemma \ref{Lemma4}, we know that if $p_i\oplus a_i=1$ holds, the state $S_i$ is likely to be the predecessor of $U_{w_i-a_i+2}$ or that of $U_{w_i-a_i+2}^{\prime}$. We go to the \textbf{if} statement in \textbf{case} 1 for a judgement. Otherwise, $S_i$ can be neither the predecessor of $U_{w_i-a_i+2}$ nor be that of $U_{w_i-a_i+2}^{\prime}$ by Lemma \ref{Lemma4} when $p_i\oplus a_i=0$ holds.

The \textbf{if} statement in \textbf{case} 1 decides whether the state $S_i$ is the predecessor of $U_{w_i-a_i+2}$ or that of $U_{w_i-a_i+2}^{\prime}$. If $S_i$ is the predecessor of either $U_{w_i-a_i+2}$ or $U_{w_i-a_i+2}^{\prime}$, $U_{w_i-a_i+2}$ and $U_{w_i-a_i+2}^{\prime}$ would have their predecessors interchange during the joining of $MC_{w_i-a_i+2}$ and $MC_{w_i-a_i}$, which means the states $S_i$ and $\hat{S_i}$ would have their successors interchange. By Lemma \ref{Lemma5}, $S_i$ is impossible to be the preferred state of cycle $C$ $P(C)$ or its companion $P(C)^{\prime}$ when the state $S_i$ is neither the predecessor of $U_{w_i-b_i+2}$ nor that of $U_{w_i-b_i+2}^{\prime}$. Namely, the successor of $S_i$ would not be changed in the generation of de Bruijn cycles.

The first \textbf{if} statement in \textbf{case} 0 decides whether $S_i^*=[a_{i+1}\cdots a_{i+n-1}10]$ is a run-cycle. In fact, it need a check of whether $S_i^*$ contains only one cyclic run of ones. Noticed that there is a precondition of \textbf{case} 0. That is $p_i\oplus a_i=0$. If $S_i^*=[a_{i+1}\cdots a_{i+n-1}10]$ is a run-cycle, $[a_{i+1}a_{i+2}\cdots a_{i+n-1}]$ must be something like this:
$$[00\cdots 0\overbrace{11\cdots 1}^{2k}],$$
where $2k$ is decided by $p_i\oplus a_i=0$. Assume that $(a_{i+1},a_{i+2},\cdots,a_{i+n-1},1)=P(C)=(0^r1^t0a_1\cdots a_{n-t-r-2}1)$, we quickly have the following contradiction:
$$E^*(C)=\underbrace{(0^r1^t0a_1\cdots a_{n-t-r-2}1)}_{at\ least\ two\ runs\ of\ ones}=\underbrace{(a_{i+1},a_{i+2},\cdots, a_{i+n-1},1,0)}_{only\ a\ run\ of\ ones}.$$
Hence, $(a_{i+1},a_{i+2},\cdots,a_{i+n-1},1)$ can never be $P(C)$ or its companion $P(C)^{\prime}$ when $S_i^*=[a_{i+1}\cdots a_{i+n-1}10]$ is a run-cycle. That is to say that in the generation of $MC_{\frac{w_i+a_{i+1}-1}{2}}$ the successor of $S_i$ would stay.

The first \textbf{if} statement in \textbf{case} 0 does some judgements on the preferred state $P(C)$ of a non-run-cycle $C$. Noticed that when $S_i^*$ has at least two runs of ones, i.e. $S_i^*$ is a non-run-cycle, there is always some satisfactory $E_i^*$. At this moment, if $E_i^*=S_i^*$, i.e. $[0^r1^t0a_s\cdots a_{n-r-t+s-3}10]=[a_{i+1} \cdots a_{i+n-1}10]$, we have $P(C)=[0^r1^t0a_s\cdots a_{n-r-t+s-3}1]=(a_{i+1}, a_{i+2}, \cdots, a_{i+n-1}, 1)$ and then the state $S_{i+1}$ must be either $P(C)$ or its companion $P(C)^{\prime}$. Thus, $S_i$ is the predecessor of either $P(C)$ or $P(C)^{\prime}$, which implies that the successor of $S_i$ would be modified in the generation of $MC_{\frac{w_i+a_{i+1}-1}{2}}$. 

The function \textbf{NextBitStab} is to generate next bit for those conditions that the successor of $S_i$ would not be interchanged. The function \textbf{NextBitIntCha} is for opposite conditions.

\begin{example}
We program Algorithm {\rm \ref{Algorithm1}} on {\rm MATLAB}. We set an initial state $(0,1,1,1,1,1,1)$ for {\rm CSR}$_7$ and use the same states of $U_{2k+1}$ as above, i.e. $U_3=(1,0,0,0,0,0,1)$, $U_5=(1,1,1,0,0,0,1)$, $U_7=(1,1,1,1,1,0,1)$, then we have the following binary de Bruijn cycle
\begin{center}
\begin{tabular}{  p{16cm}}
\textnormal{1 1 0 0 1 1 0 1 1 1 0 0 1 0 1 1 1 1 0 1 0 1 0 1 1 1 0 1 1 0 1 0 1 1 0 1 1 0 0 1 1 1 0 1 0 0 1 1 1 1 0 0 0 0 1 0 1 1 0 0 0 1 0 0 1 1 0 0 1 0 0 1 0 1 0 0 1 0 0 0 1 1 0 1 0 0 0 1 0 1 0 1 0 0 0 0 1 1 0 0 0 0 0 0 0 1 0 0 0 0 0 1 1 1 0 0 0 1 1 1 1 1 0 1 1 1 1 1};
\end{tabular}
\end{center}
Apparently, it is shift equivalent to the result in Example {\rm \ref{Example1}}.
\end{example}

For each $1\le k\le \lfloor {\frac {n}{2}} \rfloor$ there are manifestly $\binom{n-1}{2k-1}$ choices of $U_{2k+1}$, which leads to possibly the same number of de Bruijn cycles. The most costly of working space is the storage for $U_{2k+1}$, around $n\lfloor {\frac {n}{2}} \rfloor+n \approx {\frac {n^2}{2}}$. Besides, $n$ cyclic shifts at most and the same number of $n$-bit comparisons are needed to generate next bit.

\section{FSRs whose periods divide (n+1)}
Algorithm \ref{Algorithm1} lies on the fact that extended weight of any cycle keeps constant.

\begin{lemma}\label{Lemma6}
A necessary and sufficient condition for the extended weight of a given cycle produced by an $n$-stage {\rm FSR} to be constant is that period of the cycle divides $(n+1)$.
\end{lemma}

\begin{proof}
Let a cycle $C=(S_1, S_2, \cdots, S_l)=(a_1a_2\cdots a_l)$ with length $l$ and $S_i=(a_i, a_{i+1}, \cdots, a_{i+n-1})$ for any $i$, $1\le i \le l$. Here all subscripts must be reduced $mod$ $l$. Thus two extended representations of $C$ are given by $E(C)=(a_i, a_{i+1}, \cdots, a_{i+n})$ and $E(C)^\prime=(a_{i+1}, a_{i+2}, \cdots, a_{i+n+1})$. For the necessity: Since $W_E(C)=wt(a_i, a_{i+1}, \cdots, a_{i+n})$, $W_E(C)^\prime=wt(a_{i+1}, a_{i+2}, \cdots, a_{i+n+1})$ and $W_E(C)=W_E(C)^\prime$, we have $a_i=a_{i+n+1}$, then $l\ |\ (n+1)$. Now for the sufficiency: Since $l\ |\ (n+1)$, $a_i=a_{i+n+1}$ and $wt(a_i, a_{i+1}, \cdots, a_{i+n})=wt(a_{i+1}, a_{i+2}, \cdots, a_{i+n+1})$ ($1\le i \le l$), $W_E(C)$ is constant on cycle $C$.
\end{proof}

Lemma \ref{Lemma6} implies that FSRs the idea of extended representation can be generalized to should be within the class of $n$-stage FSRs (denoted by $\varOmega$) whose periods of all cycles are factors of $(n+1)$.

 \begin{theorem}\label{Theorem1}
 Let $f=x_1+g(x_2, x_3, \cdots, x_n)$ be in $\varOmega$, we have

 (1) $g(x_2, x_3, \cdots, x_n)$ is symmetric.

 (2) Either $g_k=0$ for odd $k$ and $g_k=1$ for even $k$ or $g_k=1$ for odd $k$ and $g_k=0$ for even $k$.
 \end{theorem}

 \begin{proof}
 For any cycle $C$ with length $l$ in $G_{f(n)}$, without the loss of generality, let $C=(S_1, S_2, \cdots, S_l)=(a_1a_2\cdots a_l)$. Since $l\ |\ (n+1)$, we have $$S_i=(a_i, a_{i+1}, \cdots, a_{l-1}, a_0, a_1, \cdots, a_{l-1}, \cdots, a_0, a_1, \cdots, a_{i-2})$$ for any $i \in \{ 0, 1, \cdots, l-1\}$.
By Lemma \ref{Lemma6}, let $W_E(C)=k_0$ with $0\le k_0\le(n+1)$. We then denote $g_k(x_1, x_2, \cdots, x_{n-1})$ as $g_k(m)$ $(0 \le k \le k_0)$ where $m=\sum\limits_{j=1}^{n-1}x_j2^{n-1-j}$ and $0\le m\le (2^{n-1}-1)$. For any $S_i$, it follows that $a_{i-1}=a_i\oplus g_{k_0-a_{i-1}-a_i}(m_i)$.

Case 1: $a_{i-1}+a_i=0$, $(a_{i-1}, a_i)=(0, 0)$, then $g_{k_0}(m_i)=0$;

Case 2: $a_{i-1}+a_i=1$, $(a_{i-1}, a_i)=(0, 1)$ or $(1, 0)$, then $g_{k_0-1}(m_i)=1$;

Case 3: $a_{i-1}+a_i=2$, $(a_{i-1}, a_i)=(1, 1)$, then $g_{k_0-2}(m_i)=0$.

Obviously (if $g_{k_0}$, $g_{k_0-1}$ or $g_{k_0-2}$ exists on $C$) $g_{k_0}$, $g_{k_0-1}$ and $g_{k_0-2}$ are all constant functions on $C$. Thus we denote $g_k$ on cycle $C_i$ as $g_k(i)$ in further proofs. It is also inferred that $g_{k_1}(i)=g_{k_2}(i)\oplus 1$ for odd $(k_1-k_2)$ and $g_{k_1}(i)=g_{k_2}(i)$ for even $(k_1-k_2)$ for any $g_{k_1}$ and $g_{k_2}$ on cycle $C_i$. (Note that $k_1-k_2 \in \{ -2, -1, 0, 1, 2\}$.) Equivalently, the relationship between $g_{k_1}(i)$ and $g_{k_2}(i)$ could be denoted by
\begin{equation}\label{Equation3}
g_{k_1}(i)=g_{k_2}(i)+k_1-k_2 \ (\bmod 2).
\end{equation}

Let a graph $\Delta$ be constructed as follows. Each cycle $C_i$ in $G_{f(n)}$ is represented by a vertex $V_i$ in the graph and an edge is drawn between two vertices if and only if the two corresponding cycles are adjacent. It has been proved that the graph for any nonsingular feedback function is connected. Also we remind readers that for any two adjacent cycles $C_i$ and $C_j$, let $S=(a_1, a_2, \cdots, a_n)$ on $C_i$ and $\hat{S}=(a_1\oplus 1, a_2, \cdots, a_n)$ on $C_j$ with $wt(a_2, a_3, \cdots, a_n)=k$, it follows that $g_k(i)=g_k(j)=g(a_2, a_3, \cdots, a_n)$. Therefore we denote the corresponding edge between $V_i$ and $V_j$ as $E_{i, j}(k)$. From the connectivity of the graph, we are sure to find a path for corresponding vertices of any two nonadjacent cycles $C_i$ and $C_j$ with the existence of both $g_k(i)$ and $g_k(j)$. Assume the path is represented by a set of vertices in order as $\{V_i=V_{i_1}, V_{i_2}, \cdots, V_{i_l}=V_j\}$. Let the edge between $V_{i_s}$ and $V_{i_{s+1}}$ be $E_{i_s,i_{s+1}}(k_{i_s})$ for $1\le s\le (l-1)$. We provide a schematic diagram of $\Delta$ with the neglect of the slight differences in definitions of cycle $C$ and vertex $V$.

\begin{figure}[!h]
  \centering
  \includegraphics[scale=.4]{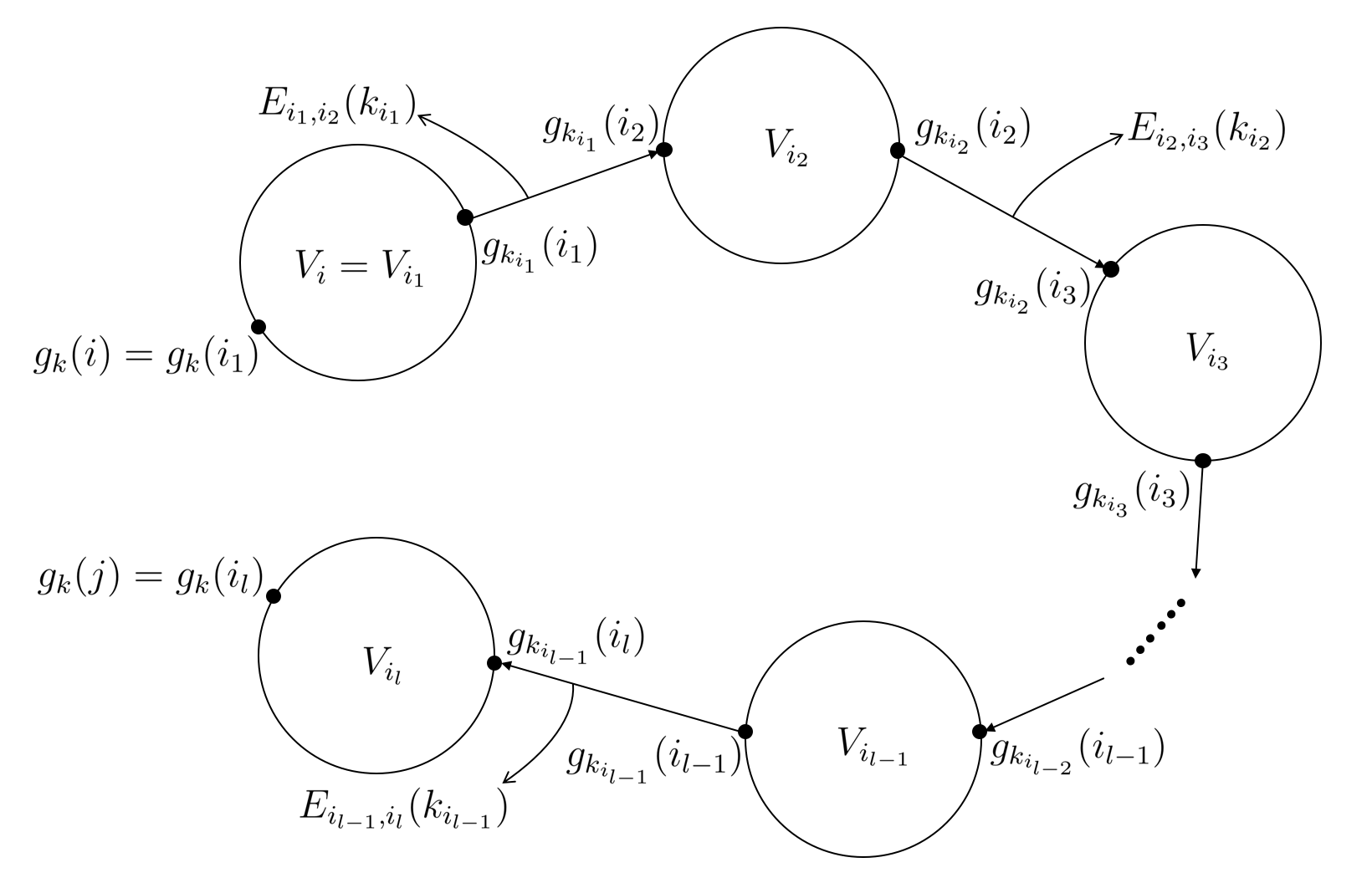}\\
\caption{$\Delta$}
\label{Delta}
\end{figure}

We obtain the relation between $g_k(i)$ and $g_k(j)$ by the equalities of $g_k$ mentioned above:
\begin{eqnarray*}
  C_{i_1}:~g_k(i_1) &=& g_{k_{i_1}}(i_1)+k-k_{i_1} \ (\bmod 2), \\
  E_{i_1,i_2}(k_{i_1}):~g_{k_{i_1}}(i_1)&=&g_{k_{i_1}}(i_2),\\
C_{i_2}:~g_{k_{i_1}}(i_2)&=&g_{k_{i_2}}(i_2)+k_{i_1}-k_{i_2} \ (\bmod 2),\\
E_{i_2,i_3}(k_{i_2}):~g_{k_{i_2}}(i_2)&=&g_{k_{i_2}}(i_3),\\
\vdots~~~~~~~~~&\vdots~&~~~~~~~~~~~~~~~~~~~\vdots\\
C_{i_{l-1}}:~g_{k_{i_{l-2}}}(i_{l-1})&=&g_{k_{i_{l-1}}}(i_{l-1})+k_{i_{l-2}}-k_{i_{l-1}} \ (\bmod 2),\\
E_{i_{l-1},i_l}(k_{i_{l-1}}):~g_{k_{i_{l-1}}}(i_{l-1})&=&g_{k_{i_{l-1}}}(i_l),\\
C_{i_l}:~g_{k_{i_{l-1}}}(i_l)&=&g_k(i_l)+k_{i_{l-1}}-k \ (\bmod 2).
\end{eqnarray*}

Cumulating the $l$ equations, we observe $g_k(i)=g_k(i_1)=g_k(i_l)=g_k(j)$, which implies that $g_k$ $(0\le k\le n-1)$ is a constant function in $G_{f(n)}$. Thus the output of $g$ depends on the Hamming weight of inputs only. The first assertion is proved.

The second assertion then follows directly.
 \end{proof}

Given two integers $a$ and $b$ and their 2-adic representations $a=\sum\limits_{i=1}^{n}a_i2^{n-i}$ and $b=\sum\limits_{i=1}^{n}b_i2^{n-i}$ we say $a\preceq b$ or equivalently $(a_1, a_2,\cdots, a_{n})\preceq (b_1, b_2,\cdots, b_{n})$ if and only if $a_i\le b_i$ for any $i\in \{1, 2, \cdots, n\}$. It is well known that the simplified value vector $v(f)$ and the simplified ANF vector $\lambda_f$ of a symmetric $n$-variable function $f$ are related by $\lambda_f(i)=\mathop{\oplus}\limits_{k\preceq i}v_f(k)$ for any $i\in \{0,1, \cdots, n\}$ in \cite{e}. Then the second assertion of Theorem \ref{Theorem1} leads to this conclusion.

\begin{corollary}
For any n-variable $f$ in $\varOmega$, we have either $v(f)=(1, 0, 1, 0, \cdots)$, $\lambda_f=(1, 1, 0, 0, \cdots, 0)$ or $v(f)=(0, 1, 0, 1, \cdots)$, $\lambda_f=(0, 1, 0, 0, \cdots, 0)$. Equally we have either $f=f_P=x_1\oplus x_2\oplus \cdots \oplus x_n$ or $f=f_C=x_1\oplus x_2\oplus \cdots \oplus x_n\oplus 1$.
\end{corollary}

Clearly, $\varOmega$ consists of PSR and CSR. It can be seen that generalizations of extended representation $E(C)$ and extended weight $W_E(C)$ in algorithms to generate de Bruijn sequences from FSRs is very limited.


\bibliography{mybibfile}

\begin{thebibliography}{99}
\bibitem{a}
{T. Etzion, A. Lempel},
Algorithms for the generation of full-length shift- register sequences. IEEE Trans. Inform. Theory, vol. IT-30, no. 3, pp. 480-484, May 1984.

\bibitem{b}
{S. W. Golomb},
Shift Register Sequences. San Francisco, CA: Holden-Day Inc., 1967.

\bibitem{c}
{H. Fredricksen},
 A class of nonlinear DeBruijn cycles, J. Combinat. Theory, Ser. A, vol. 19, pp. 192-199, Sept. 1975.

\bibitem{f}
 {Erik R. Hauge},
 On the cycles and adjacencies in the complementary circulating register, Discrete Mathematics, vol. 145, no. 1-3, pp. 105-132, 1995.

\bibitem{d}
{Z.X. Wang, H. Xu and W.F. Qi},
On the cycle structure of some nonlinear feedback shift registers, Chinese Journal of Electronics,
vol. 23, no. 4, pp. 801-804, 2014.

\bibitem{e}
{A. Canteaut, M. Videau},
Symmetric Boolean functions, IEEE Trans. Inform. Theory, vol. 51, no. 8, pp. 2791-2811, 2005.

\end{thebibliography}

\end{document}